\documentclass{article}

\usepackage{amsmath, amssymb}
\usepackage{graphicx, wrapfig} 
\usepackage{amsthm}
\usepackage{url}
\usepackage{color}
\usepackage{bm}
\usepackage{enumitem}

\theoremstyle{plain}
\newtheorem{theorem}{Theorem}

\theoremstyle{definition}
\newtheorem{definition}{Definition}
\newtheorem{example}{Example}

\newtheorem{proposition}{Proposition}
\newtheorem{lemma}{Lemma}
\newtheorem{corollary}{Corollary}

\title{Characterization of non-self OU sequences of two-component link diagrams}
\author{Naoki Sakata\thanks{Advanced Institute of Material Research, Tohoku University, 2-1-1 Katahira, Aoba-ku, Sendai, Miyagi, 980-8577, Japan. Email: sakata@casis.sakura.ne.jp},  
Ayaka Shimizu\thanks{Center for Soft Matter Physics, Ochanomizu University, 2-1-1 Otsuka, Bunkyo-ku, Tokyo, 112-8610, Japan. Email: shimizu.ayaka@ocha.ac.jp, shimizu1984@gmail.com} 
and Koya Shimokawa\thanks{Department of Mathematics, Ochanomizu University, 2-1-1 Otsuka, Bunkyo-ku, Tokyo, 112-8610, Japan. Email: shimokawa.koya@ocha.ac.jp}
\thanks{International Institute for Sustainability
with Knotted Chiral Meta Matter (WPI-SKCM$^2$),
Hiroshima University, 1-3-1 Kagamiyama, 
Higashi-Hiroshima, 739-8526, Japan}
\thanks{Research Institute for Science and Technology at Tokyo University of Science,  Division of Joint Research of Geometry and Natural Science, 1-3 Kagurazaka, Shinjuku-ku, Tokyo 162-8601, Japan}}
\date{\today}

\begin{document}

\maketitle

\begin{abstract}
A non-self OU sequence is a cyclic sequence of crossing information of non-self crossings that is obtained by traversing a knot component of an oriented link diagram. 
In this paper, we investigate what information can be derived from non-self OU sequences, and we completely characterize pairs of non-self OU sequences of diagrams of two-component links. 
We also characterize the pairs for specific prime links with crossing number up to five. 
\end{abstract}

\section{Introduction}

We investigate what information can be inferred from limited data of link diagrams. 
In \cite{H}, pseudo diagrams of knots, links and spatial graphs were introduced by Hanaki as a diagram with some crossing information (over/under information) missing. 
This study is motivated by the analysis of images of DNA knots in which some crossing information is unclear. 
In~\cite{M}, the complexity of each two-component link in a large collection of ring polymers is studied in the context of polymer melts. 
Images or simulation images of polymers typically have quite a large number of crossings, in particular self-crossings. 

Here, a {\it self-crossing} of a link diagram $D=K_1 \cup K_2 \cup \dots \cup K_r$ is a crossing of the same knot component $K_i$ ($i=1, 2, \dots , r$). 
For example, the link diagram shown in Figure~\ref{fig-whitehead} has a self-crossing indicated by $p$. 
We call a crossing between distinct components $K_i$ and $K_j$ a {\it non-self crossing}. 
In this paper, we ignore self-crossings and focus on non-self crossings in two-component link diagrams, considering cyclic sequences of crossing information.  \\
\begin{figure}[ht]
\centering
\includegraphics[width=3cm]{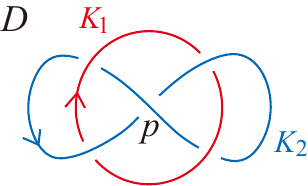}
\caption{An oriented diagram $D=K_1 \cup K_2$ of a Whitehead link. }
\label{fig-whitehead}
\end{figure}

For an oriented link diagram $D=K_1 \cup K_2 \cup \dots \cup K_r$ on the two-sphere $S^2$, when we traverse a knot component $K_i$ with the orientation, we encounter over-crossings ($O$) or under-crossings ($U$). 
By recording this crossing information, we obtain a cyclic sequence of $O$ and $U$, called the {\it OU sequence of $K_i$}. 
If we record crossing information only for non-self crossings, we obtain another cyclic sequence, and call it the {\it non-self OU sequence of $K_i$} and denote it by $\hat{f}(K_i)$. 
For an oriented link diagram $D=K_1 \cup K_2 \cup \dots \cup K_r$, we denote the tuple of the non-self OU sequences by $\hat{f}(D)=( \hat{f}(K_1), \hat{f}(K_2), \dots , \hat{f}(K_r))$. 
For example, the diagram $D$ in Figure \ref{fig-whitehead} has $\hat{f}(D)=(OUOU, OOUU)$.

There is a useful characterization for a sequence of crossing information to be realizable as a knot diagram.
Let $\# O(w)$ (resp. $\# U (w)$) denote the number of $O$s (resp. $U$s) in a cyclic sequence $w$. 
As defined in \cite{HNSY}, a cyclic sequence $w$ of $O$ and $U$ is {\it well-balanced} if $\# O(w)=\# U(w)$. 
Higa, Nakanishi, Satoh and Yamamoto proved in \cite{HNSY} that a cyclic sequence $w$ of $O$ and $U$ is realizable as an OU sequence for an oriented knot diagram $D$ if and only if $w$ is well-balanced.


We give a similar characterization for non-self OU sequences of two-component oriented link diagrams.
Here, we define a pair of non-self OU sequences to be ``well-balanced'' as follows.
\begin{definition}
A pair of cyclic sequences $(w_1, w_2)$ of $O$ and $U$ is {\it well-balanced} if it satisfies the following conditions (i) and (ii). 
\begin{itemize}
\item[(i)] The length of $w_1$ is even. 
\item[(ii)] $\# O(w_1)= \# U(w_2)$ and $\# U(w_1)= \# O(w_2)$. 
\end{itemize}
\label{def-well}
\end{definition}
\medskip 

\noindent If $(w_1, w_2)$ is well-balanced, then the length of $w_2$ is equal to that of $w_1$ by (ii).
Let $\mathcal{L}^{\text{odd}}$ (resp. $\mathcal{L}^{\text{even}}$) be the set of all oriented diagrams of two-component links with odd (resp. even) linking number. 
We put $\mathcal{L} = \mathcal{L}^{\text{odd}} \cup \mathcal{L}^{\text{even}}$.
For a set $\mathcal{S}$ of link diagrams, let $\hat{f}(\mathcal{S}) = \{ \hat{f}(D) \ | \ D \in \mathcal{S} \}$. 

We have the following characterization, which is proved in Section \ref{section-OU-sequence}. 
\begin{theorem}
For pairs $(w_1, w_2)$ of cyclic sequences of $O$ and $U$, the following hold.  
\begin{enumerate}[label=(\Roman*)]
\item $\hat{f}(\mathcal{L})= \{ (w_1, w_2) \mid (w_1, w_2)\text{: well-balanced}\, \}$
\item $\displaystyle \hat{f}(\mathcal{L}^{\text{even}})= \{ (w_1, w_2) \mid (w_1, w_2)\text{: well-balanced and }  \# O(w_1) \equiv 0 \pmod{2} \}$.
\item $\displaystyle \hat{f}(\mathcal{L}^{\text{odd}})= \{ (w_1, w_2) \mid (w_1, w_2)\text{: well-balanced and }  \# O(w_1) \equiv 1 \pmod{2} \}$.
\end{enumerate}
\label{thm-ch}
\end{theorem}
\medskip 

\noindent Let $\mathcal{L}^T$ be the set of all oriented diagrams of a trivial two-component link. 
By Theorem \ref{thm-ch}, we have $\hat{f} (\mathcal{L}^T ) \subseteq \hat{f}(\mathcal{L}^{\text{even}})$. 
Moreover, we prove the following theorem in Section \ref{section-proof-ch2}. 

\medskip 
\begin{theorem}
We have $\hat{f}(\mathcal{L}^T) = \hat{f}(\mathcal{L}^{\text{even}})$.
\label{thm-ch2}
\end{theorem}
\medskip 

\noindent In this paper, we consider links up to the mirror image. 
We also characterize pairs of non-self OU sequences for two-component links called a Hopf link, Solomon link, and Whitehead link and denoted by $2^2_1$, $4^2_1$, and $5^2_1$ in Rolfsen's knot table \cite{R}. 
Let $\mathcal{L}^H$, $\mathcal{L}^S$ and $\mathcal{L}^W$ be the sets of all oriented diagrams of the links $2^2_1$, $4^2_1$ and $5^2_1$. 
The following theorem is proved in Section \ref{section-proof-ch3}.

\medskip 
\begin{theorem}
The following hold. 
\begin{enumerate}[label=(\Alph*)]
\item $\hat{f}(\mathcal{L}^H) = \hat{f}(\mathcal{L}^{\text{odd}})$.
\item $\hat{f}(\mathcal{L}^S) = \hat{f}(\mathcal{L}^W) = \left\{ (w_1, w_2) \;\middle|\; \begin{aligned}
      & (w_1, w_2)\text{: well-balanced,} \\
      & \#O(w_1) \equiv 0 \pmod{2}, \\
      & \#O(w_1) \geq 2 \text{ and } \#U(w_1) \geq 2
    \end{aligned}
  \right\}.$ 
\end{enumerate}
\label{thm-ch3}
\end{theorem}
\medskip 

\noindent Let $T(2,2n)$ be a $(2,2n)$-torus link with a positive integer $n$. 
Let $\mathcal{L}^{T(2,2n)}$ be the set of all oriented diagrams of $T(2,2n)$. 
The following theorem implies that there exist infinitely many distinct sets of pairs of non-self OU sequences. 

\medskip 
\begin{theorem}
If $n \neq m$, then $\hat{f}(\mathcal{L}^{T(2,2n)}) \neq \hat{f}(\mathcal{L}^{T(2,2m)})$. 
\label{thm-ch4}
\end{theorem}
\medskip

\noindent In the proofs of Theorems \ref{thm-ch2} and \ref{thm-ch3}, we use the ``OU number'' of non-self OU sequences, which was recently introduced by the authors in \cite{SSS} in the study of the number of Reidemeister moves of type III. 
In Section \ref{section-OU-number}, we give a geometrical meaning of the OU number (Proposition \ref{prop-lk-Phi}). 
Using it, we also prove the following proposition in Section \ref{section-OU-number}. 

\medskip 
\begin{proposition}
Let $(w_1, w_2)$ be a pair of cyclic sequences of $O$ and $U$. 
If the OU numbers of $w_1$ and $w_2$ are not the same, then any link diagram $D$ with $\hat{f}(D)=(w_1, w_2)$ has a self-crossing. 
\label{prop-self-crossing}
\end{proposition}
\medskip 

\noindent Proposition \ref{prop-self-crossing} implies that we can obtain information of not only non-self crossings but also self-crossings of a diagram from the pair of non-self OU sequences. \\

The rest of the paper is organized as follows. 
In Section \ref{section-OU-sequence}, we review the OU sequence and prove Theorems \ref{thm-ch} and \ref{thm-ch4}. 
In Section \ref{section-OU-number}, we recall the OU number and study the relation to the linking number. 
In Section \ref{section-proof-ch2}, we prove Theorem \ref{thm-ch2}. 
In Section \ref{section-proof-ch3}, we prove Theorem \ref{thm-ch3}.

\section{OU sequences and non-self OU sequences}
\label{section-OU-sequence}

In this section, we study the OU sequences and prove Theorem \ref{thm-ch}. 
In Section \ref{subsection-OU-k}, we review the OU sequence of knot diagrams. 
In Section \ref{subsection-OU-l}, we recall the non-self OU sequence for link diagrams and observe the relation to the linking warping degree and splitting number. 
In Section \ref{subsection-well}, we consider well-balanced pairs of sequences and prove Theorem \ref{thm-ch}.

\subsection{OU sequence of knot diagrams}
\label{subsection-OU-k}

We review the OU sequence of knot diagrams. 
Let $D$ be an oriented knot diagram on $S^2$. 
By traversing $D$ and recording the crossing information at each crossing, we obtain a cyclic sequence consisting of the letters ``$O$'' and ``$U$'', where ``$O$'' denotes an over-crossing and ``$U$'' denotes an under-crossing. 
We refer to this as the {\it OU sequence of $D$} and denote it by $f(D)$. 
For example, the oriented knot diagram $D$ shown in Figure \ref{fig-OU-k} has the OU sequence $f(D)=OUUOOUUUOUOOOU$. 
\begin{figure}[ht]
\centering
\includegraphics[width=3cm]{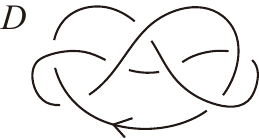}
\caption{A knot diagram $D$ with $f(D)=OUUOOUUUOUOOOU$.}
\label{fig-OU-k}
\end{figure}
We can also obtain the OU sequence from the Dowker-Thistlethwaite code or the Gauss code of a knot diagram, while we cannot recover these codes from the OU sequence since the pairing information is missing. \\

In this section, $O^a$ (resp. $U^b$) in a sequence denotes $a$ consecutive $O$s (resp. $b$ consecutive $U$s). 
Let $\mathcal{K}$ be the set of all oriented knot diagrams. 
Let $\mathcal{U}$ (resp. $\mathcal{N}$) be the set of all diagrams of the trivial knot (resp. nontrivial knots). 
Let $\mathcal{T}$ be the set of all diagrams of a trefoil knot. 
For each set $S$ of knot diagrams, put $f(S) = \{ f(D) \ | \ D \in S \}$. 
Recall that a sequence $w$ of $O$ and $U$ is said to be well-balanced if $\# O(w)=\# U(w)$. 
In \cite{HNSY}, the OU sequences of knot diagrams were characterized as follows:

\begin{theorem}[\cite{HNSY}]
The following hold. 
\begin{enumerate}[label=(\arabic*)]
\item $f(\mathcal{K})= f(\mathcal{U})= \{ w \ | \ w \text{: well-balanced } \}$.
\item $f(\mathcal{N})\subsetneq f(\mathcal{K})$.
Furthermore, $w \in f(\mathcal{K})\setminus f(\mathcal{N})$ if and only if $w$ is the empty sequence, or is of the form $O^aU^a$, $OU^a O^{a+b-1}U^b$, $O^aUO^bU^{a+b-1}$, $OU^aO^{a+2b-1}UO^cU^{2b+c-1}$, or $OU^{a+2b-1}O^aUO^{2b+c-1}U^c$ for some positive integers $a$, $b$, and $c$.
\item $f(\mathcal{N})=f(\mathcal{T})$.
\end{enumerate}
\end{theorem}

\subsection{The non-self OU sequence of link diagrams}
\label{subsection-OU-l}

In this subsection, we review the non-self OU sequences of link diagrams. 
Let $D=K_1 \cup K_2 \cup \dots \cup K_r$ be an oriented link diagram on $S^2$. 
By traversing a knot component $K_i$ and recording the crossing information at each {\it non-self crossing}, we obtain a cyclic sequence of $O$ and $U$. 
We refer to this as the {\it non-self OU sequence of $K_i$} and denote it by $\hat{f}(K_i)$. 
For the length of $\hat{f}(K_i)$, the following proposition holds. 

\medskip
\begin{proposition}[\cite{SSS}]
Let $D= K_1 \cup K_2 \cup \dots \cup K_r$ be an oriented link diagram. 
The length of the non-self OU sequence $\hat{f}(K_i)$ is an even number for each $K_i$. 
\label{prop-length}
\end{proposition}
\medskip

In the rest of this section, we focus on two-component links. 
Recall that $\# O(w)$, $\# U(w)$ denote the number of $O$s, $U$s in a sequence $w$, respectively. 
For the relation to the linking number $lk(D)$, the following lemma was shown in \cite{SSS}. 

\medskip 
\begin{lemma}[\cite{SSS}]
For any oriented diagram $D=K_1 \cup K_2$ with $\hat{f}(D)=(w_1, w_2)$, the following hold. 
\begin{itemize}
\item[$\bullet$] $| lk(D) | \leq \min \{ \# O(w_1), \# U(w_1) \}$,  
\item[$\bullet$] $lk(D) \equiv \# O(w_1) \pmod{2}$.
\end{itemize}
\label{lem-lk-mod}
\end{lemma}

\begin{example}
Let $D$ be an oriented diagram of a two-component link with $\hat{f}(D)=(OUO^2U^4, O^5U^3)$. 
From Lemma \ref{lem-lk-mod}, we infer that $lk(D)= \pm1$ or $\pm3$. 
\end{example}
\medskip

The ``linking warping degree'', $ld(D)$, of a two-component link diagram $D=K_1 \cup K_2$ is defined in \cite{Sw} as $ld(D)= \min \{ \# U(K_1), \# U(K_2) \}$, where $\# U(K_i)$ indicates the number of non-self crossings of $D$ at which $K_i$ is under. 
The linking warping degree is a restricted version of the ``warping degree'' (\cite{K}) of oriented link diagrams which measures a complexity of link diagrams. 
For oriented diagrams $D$ of two-component links, we have $ld(D)= \min \{ \# O(w_1), \# U(w_1) \}$ for $\hat{f}(D)=(w_1, w_2)$ by definition. 

The {\it splitting number}\footnote{This splitting number is different from the ``splitting number'' studied in \cite{A}, \cite{CCC}, \cite{Sc}, etc., which permits self-crossing changes as well. 
It is also called the ``weak splitting number'' (\cite{CCC, LI}) and denoted by $wsp(L)$ to avoid confusion. }, $sp(L)$, of a link $L$ is the minimal number of crossing changes {\it between different components} that are required to obtain a split link from $L$. 
From the inequality $sp(L) \leq ld(D)$ shown in \cite{Sw}, we obtain the following proposition. 

\medskip
\begin{proposition}
For any oriented diagram $D$ of a two-component link $L$ with $\hat{f}(D)=(w_1, w_2)$, we have 
$sp(L) \leq \min \{ \# O(w_1), \# U(w_1) \}$. 
\label{prop-sp}
\end{proposition}
\medskip 

\noindent Lemma \ref{lem-lk-mod} and Proposition \ref{prop-sp} are effective to characterize the non-self OU sequences for individual links. 
The splitting number has been actively studied and various techniques and lower bounds have been developed (see, for example, \cite{BS, CFP, CCZ, L}). 
The values of this invariant have been determined for all prime links with crossing number up to nine\footnote{The weak splitting number $wsp(L)$ has also been determined for all prime links with crossing number up to nine (see \cite{CCC, LI}).} (see \cite{CFP, LI}). 

For larger crossing numbers, a number of links have also been determined. 
For example, the splitting numbers of the links $L12n1367$, $L11a372$ are determined to be 3, 5 in \cite{BS}, \cite{CFP}, respectively (see also \cite{CCZ}). 
By Proposition \ref{prop-sp}, we can estimate the number of occurrences of $O$s and $U$s in the non-self OU sequences.

\medskip 
\begin{example}
Applying this estimation to the specific links mentioned above, we have the following. 
\begin{itemize}
\item[$\bullet$] For any oriented diagram $D=K_1 \cup K_2$ of the link $L12n1367$ with $\hat{f}(D)=(w_1, w_2)$, we have $\# O(w_1), \# U(w_1) \geq 3$. 
\item[$\bullet$] For any oriented diagram $D=K_1 \cup K_2$ of the link $L11a372$ with $\hat{f}(D)=(w_1, w_2)$, we have $\# O(w_1), \# U(w_1) \geq 5$. 
\end{itemize}
Note that the linking numbers of these links are 1 with a suitable orientation. 
\end{example}
\medskip 

\noindent We prove Theorem \ref{thm-ch4}. 

\medskip 
\begin{proof}[Proof of Theorem \ref{thm-ch4}]
Without loss of generality, we assume that $n<m$. 

Let $w=OUO \dots U$ be an alternating sequence, namely, a cyclic sequence with no consecutive ``$OO$'' or ``$UU$'', of length $2n$. 
We obtain that $(w,w) \in \hat{f}(\mathcal{L}^{T(2,2n)})$ since the minimal crossing diagram $D_n$ of $T(2,2n)$ satisfies $\hat{f}(D_n)=(w,w)$. 
Note that the splitting number of $T(2, 2m)$ is $m$.\footnote{By the ``linking number bound'' for two-component links (see, for example, \cite{CCZ}), we have the lower bound $m = | lk(T(2,2m))| \leq sp(T(2,2m))$. Furthermore, the minimal crossing diagram $D_m$ of $T(2,2m)$ provides the upper bound $sp(T(2,2m)) \leq ld(D_m)=m$. Thus, $sp(T(2, 2m)) = m$.}
By Proposition~\ref{prop-sp}, we have $\# O(w_1), \# U(w_1) \geq m$ for any diagram $D$ of $T(2,2m)$ with $\hat{f}(D)=(w_1, w_2)$. 
Then we obtain $(w,w) \not\in \hat{f}(\mathcal{L}^{T(2,2m)})$ and therefore $\hat{f}(\mathcal{L}^{(T(2,2n)}) \neq \hat{f}(\mathcal{L}^{T(2,2m)})$. 
\end{proof}

\subsection{Well-balanced pairs of sequences}
\label{subsection-well}

In this subsection, we study well-balanced pairs of cyclic sequences and prove Theorem \ref{thm-ch}. 
The following lemma corresponds to the statement (I) in Theorem~\ref{thm-ch}. 

\medskip 
\begin{lemma}
Let $(w_1, w_2)$ be a pair of cyclic sequences of $O$ and $U$. 
Then $(w_1, w_2)$ is realizable as $(w_1, w_2) = \hat{f}(D)$ for some oriented diagram $D$ of a two-component link if and only if $(w_1, w_2)$ is well-balanced. 
\label{lem-1}
\end{lemma}
\medskip 

\begin{proof}
Let $(w_1, w_2)$ be a pair of cyclic sequences of $O$ and $U$. 
\begin{itemize}
\item[($\Rightarrow$): ] Suppose that $(w_1, w_2)=(\hat{f}(K_1), \hat{f}(K_2))$ for a two-component link diagram $D=K_1 \cup K_2$. 
\begin{itemize}
\item[(1):] The length of $w_1 = \hat{f}(K_1)$ is an even number by Proposition \ref{prop-length}. 
\item[(2):] A non-self crossing $p$ is an over-crossing (resp. under-crossing) of $K_1$ if and only if $p$ is an under-crossing (resp. over-crossing) of $K_2$. 
This implies that $\# O(w_1)= \# U(w_2)$ (resp. $\# U(w_1)= \# O(w_2)$). 
\end{itemize}
Therefore, $(w_1, w_2)$ is well-balanced. 
\item[($\Leftarrow$): ] Suppose that $(w_1, w_2)$ is well-balanced. 
Then $w_1 = \emptyset$ if and only if $w_2 = \emptyset$ by the condition (ii). 
Any diagram $D=K_1 \cup K_2$ without non-self crossings satisfies $(\hat{f}(K_1), \hat{f}(K_2))=(\emptyset , \emptyset )$. 
Suppose that $(w_1, w_2) \neq (\emptyset, \emptyset)$. 
We construct a diagram $D=K_1 \cup K_2$ with $\hat{f}(D)=(w_1, w_2)$ in the following 2 steps. 
\begin{itemize}
\item[Step 1. ] Take $\# O(w_1) + \# U(w_1)$ points on an oriented simple closed curve $C$ on $S^2$ and assign the crossing information according to $w_1$. 
Note that $C$ has $\# O(w_1)$ over crossings and $\# U(w_1)$ under crossings. 
\item[Step 2. ] Connect the crossings on $C$ by a curve that crosses $C$ transversely, according to $w_2$ without producing extra non-self crossings, while self-crossings may be produced. 
We can draw a curve that realizes the crossing information of $w_2$ because of the condition (ii). 
We can also draw a closed curve by condition (i). 
\end{itemize}
We refer to this procedure as {\it Construction 1}. 
Thus, we obtain a diagram $D$ such that $\hat{f}(D)=(w_1, w_2)$. 
\end{itemize}
\end{proof}
\medskip 

\noindent See Example \ref{ex-const-w1} for Construction 1. 

\medskip 
\begin{example}
Let $(w_1, w_2)=(O^2UO^2U, \ OU^3OU)$. 
Construction 1 is shown in Figure \ref{fig-const-w1}. 
The lower-right diagram $D=K_1 \cup K_2$ in Figure \ref{fig-const-w1} satisfies $\hat{f}(D)=(w_1, w_2)$ with arbitrary crossing information at the self-crossing of $K_2$. 
\begin{figure}[ht]
\centering
\includegraphics[width=12cm]{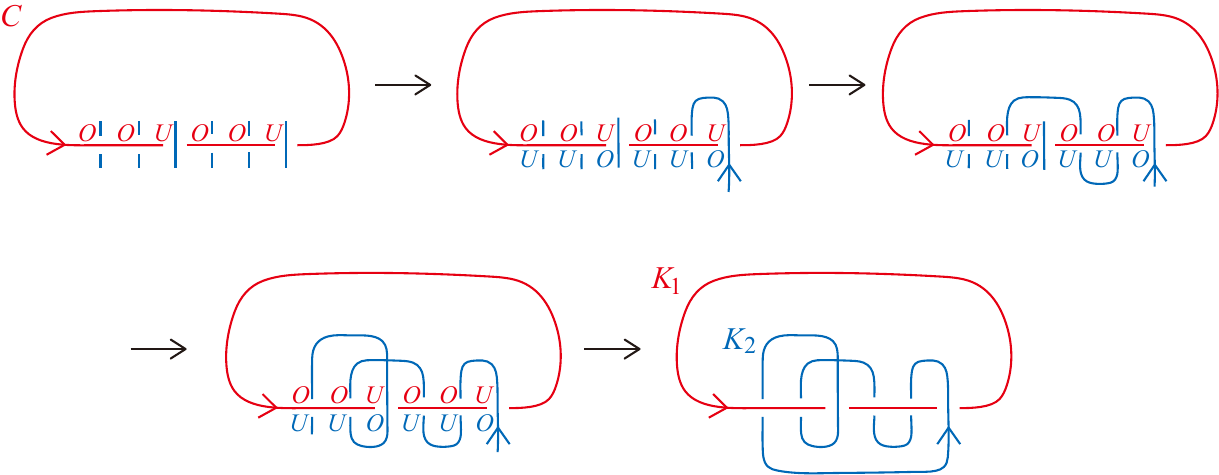}
\caption{Construction 1. The upper-left shows Step 1 and the rest show Step 2.}
\label{fig-const-w1}
\end{figure}
\label{ex-const-w1}
\end{example}
\medskip 

\noindent ``Construction 2'', the method of creating a diagram based on $w_2$, is also effective, as shown in the following example. 

\medskip 
\begin{example}
Let $(w_1, w_2)=(O^2UO^2U, \ OU^3OU)$. 
Diagrams $D=K_1 \cup K_2$ with $\hat{f}(D)=(w_1, w_2)$ are obtained by Construction 2 as shown in Figure \ref{fig-const-w2}. 
Note that a diagram obtained by Construction 1 or 2 is not unique. 
\begin{figure}[ht]
\centering
\includegraphics[width=11cm]{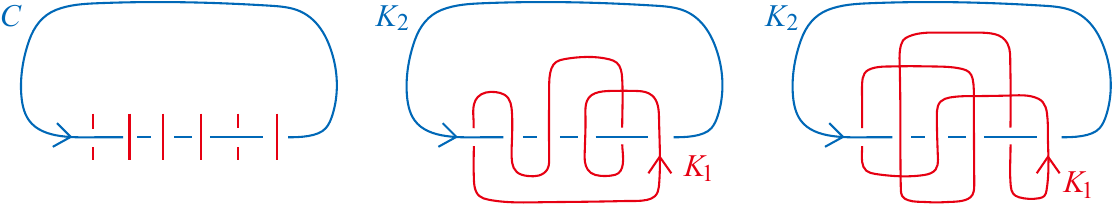}
\caption{Construction 2. The diagram on the left-hand side shows Step 1. The other two diagrams show results of Step 2 with arbitrary crossing information at the self-crossings of $K_1$. }
\label{fig-const-w2}
\end{figure}
\end{example}
\medskip 

\noindent In Section \ref{section-OU-number}, we study ``OU number'' of non-self OU sequences, and show that the linking number of any diagram that is obtained by Construction 1 (resp. 2) for $(w_1, w_2)$ coincides with the OU number of $w_2$ (resp. $w_1$). 
Now we prove Theorem \ref{thm-ch}. \\

\noindent {\it Proof of Theorem \ref{thm-ch}.} \ 
Statement (I) follows immediately from Lemma \ref{lem-1}. 
Statements (II) and (III) follow from (I) and Lemma \ref{lem-lk-mod}. 
\hfill$\square$  \\

\section{The OU number and linking number}
\label{section-OU-number}

The OU number is an essential concept to prove Theorem \ref{thm-ch2}. 
In this section, we review and study the OU number.

\subsection{The OU number of sequences}

The OU number of non-cyclic sequences is defined in \cite{SSS} as follows.  

\medskip 
\begin{definition}[\cite{SSS}]
Let $S$ be a non-cyclic sequence of $O$ and $U$. 
Let $N^O_o$ (resp. $N^O_e$) be the number of $O$s in $S$ that are in odd (resp. even) position. 
Let $N^U_o$ (resp. $N^U_e$) be the number of $U$s in $S$ that are in odd (resp. even) position. 
The {\it OU number of $S$}, denoted by $\Phi (S)$, is the value $\frac{1}{2}(N^O_e + N^U_o - N^O_o - N^U_e)$. 
\end{definition}
\medskip 

\noindent The OU number of cyclic sequences is defined as follows. 

\medskip 
\begin{definition}[\cite{SSS}]
Let $w$ be a cyclic sequence of $O$ and $U$ of even length. 
The {\it OU number of $w$}, denoted by $\Phi (w)$, is defined by $\Phi (w)= | \Phi (S)|$, where $S$ is a non-cyclic sequence obtained from $w$ by choosing any initial letter. 
\end{definition}
\medskip 

\begin{example}
We have $\Phi (S)=-1$ for a non-cyclic sequence $S=OUOUUO$ since $S$ has $N^O_e = N^U_o =1$ and $N^O_o = N^U_e =2$. 
We have $\Phi (w)=1$ for the cyclic sequence $w=OUOUUO$. 
\end{example}
\medskip 

\noindent By definition, we have the following. 

\medskip 
\begin{proposition}[\cite{SSS}]
Let $w$ be a cyclic sequence of $O$ and $U$ with an even length $m$. 
Then $\Phi (w) \leq \frac{m}{2}$. 
The equality holds if and only if $w$ is an alternating sequence. 
\label{prop-OU-length}
\end{proposition}
\medskip 

\noindent The following proposition was also shown in \cite{SSS}. 

\medskip 
\begin{proposition}[\cite{SSS}]
Let $w$ be a cyclic sequence of $O$ and $U$ of even length. 
We have $\Phi (w) \equiv \# O(w)  \pmod{2}$.
\label{prop-Phi-O}
\end{proposition}
\medskip 

\noindent When $(w_1, w_2)$ is well-balanced, it follows that $\# O(w_1) \equiv \# U(w_1) \equiv \# O(w_2) \equiv \# U(w_2) \equiv \Phi (w_1) \equiv \Phi (w_2) \pmod{2}$ from Proposition \ref{prop-Phi-O}. 

\medskip 
\begin{example}
For the diagram $D=K_1 \cup K_2$ in Figure \ref{fig-whitehead}, the pair $(w_1, w_2)=( \hat{f}(K_1), \hat{f}(K_2))=(OUOU, OOUU)$ is well-balanced and we have $\#O (w_1)= \#U(w_1)= \#O(w_2)= \# U(w_2)=\Phi (w_1)=2$, $\Phi (w_2)=0$. 
\end{example}

\medskip 
\begin{definition}[\cite{SSS}]
For a cyclic or non-cyclic sequence of $O$ and $U$, a {\it reduction} is an operation removing ``$OO$'' or ``$UU$''. 
\end{definition}
\medskip 

\noindent As shown in \cite{SSS}, the value of the OU number is preserved under the reduction, and the OU number can also be calculated as follows. 
For a cyclic sequence $w$ of $O$ and $U$ of even length, apply reductions in any order to obtain an alternating sequence $w'$. 
The OU number of $w$ is the number of the pairs ``$OU$'' in $w'$. 

\medskip 
\begin{example}
For a cyclic sequence $w=OUOUUO$, we can calculate the OU number by reductions as $\Phi (OUO\boldsymbol{UU}O)= \Phi (OU\boldsymbol{OO}) = \Phi (OU)=1$.
\end{example}
\medskip 

\noindent In the proof of Lemma \ref{lem-trivial} in Section \ref{section-proof-ch2}, we give a method to create a diagram of a trivial link for a given well-balanced pair $(w_1, w_2)$ with $\# O(w_1) \equiv 0 \pmod{2}$. 
In the method, we use reductions, and the obtained alternating sequence forms a core part of the diagram. 
The following lemma will be used in the proof of Lemma \ref{lem-trivial} where we need some ``$OU$'' pairs that survive reductions. 

\medskip 
\begin{lemma}
Let $w$ be a cyclic sequence of $O$ and $U$ of even length containing both symbols. 
\begin{enumerate}[label=(\arabic*)]
\item There are at least $\Phi (w)$ ``$OU$'' pairs in $w$. 
\item There is a sequence of reductions to transform $w$ into an alternating sequence $w'$ of length $2 \Phi (w)$ such that $w$ has $\Phi (w)$ pairs of ``$OU$'' that remain in $w'$. 
\end{enumerate}
\label{lem-seq-1}
\end{lemma}

Before proving Lemma 3, we make a brief remark.
As shown in the following example, if we distinguish the letters in the original sequence, the letters that remain after the reduction of pairs depend on the choice of pairs selected for the reduction. 

\begin{example}
Let $w=OOOUOUUU$. 
Then $\Phi (w)=2$, and $w$ has two pairs of ``$OU$''. 
For convenience, we label the letters as $w=O_1 O_2 O_3 U_4 O_5 U_6 U_7 U_8$.

If we apply reductions at ``$O_2 O_3$'' and ``$U_6 U_7$'', we obtain an alternating sequence ``$O_1 U_4 O_5 U_8$''. 
Neither the ``$OU$'' pairs ``$O_1 U_4$'' nor ``$O_5 U_8$'' is consecutive in $w$. 
On the other hand, if we apply reductions at ``$O_1 O_2$'' and ``$U_7 U_8$'',  we obtain ``$O_3 U_4 O_5 U_6$''. 
Both ``$OU$'' pairs ``$O_3 U_4$'' and ``$O_5 U_6$'' are consecutive in $w$. 
\end{example}


\begin{proof}[Proof of Lemma~\ref{lem-seq-1}]
It is sufficient to prove (2) since (1) follows from (2). 
Let $w$ be a cyclic sequence of $O$ and $U$ of even length. 
Let $k$ be the number of reductions required to transform $w$ into an alternating sequence. 
Namely, $k= \frac{1}{2}(|w|-2 \Phi (w))$, where $|w|$ denotes the length of $w$. 
We prove the claim (2) by an induction on $k$. 

When $k=0$, this holds since $w$ is alternating. 
Assume that the claim holds when $k=l$. 
Suppose that $k=l+1$. 
Apply a reduction on $w$, where we remove the first and second $O$s in consecutive $O$s or the first and second $U$s from the last in consecutive $U$s. 
Then we obtain a sequence $w^*$ such that $k=l$. 
By assumption, $w^*$ has $\Phi (w^*)$ ($=\Phi (w)$) pairs of ``$OU$'' that remain in an alternating sequence obtained by a sequence of reductions. 
The same ``$OU$'' pairs are in the original sequence $w$, too.
\end{proof}

Note that the specific reduction method used in this proof will be applied later in the proof of Lemma~\ref{lem-trivial}.

\subsection{OU number for link diagrams}

In this subsection, we study the OU number for link diagrams. 

\medskip 
\begin{definition}[\cite{SSS}]
Let $D=K_1 \cup K_2 \cup \dots \cup K_r$ be an oriented link diagram. 
The {\it OU number of $K_i$}, denoted by $\Phi (K_i)$, is defined as $\Phi ( \hat{f}(K_i))$. 
\end{definition}
\medskip

\noindent It was shown in \cite{SSS} that the OU number does not depend on the orientation of the link diagram. 
Therefore, the OU number can be defined for unoriented link diagrams as well. 
We have the following proposition, which describes a geometrical meaning of the OU number. 

\medskip 
\begin{proposition}
Let $D=K_1 \cup K_2$ be an oriented link diagram such that $K_1$ has no self-crossings. 
Then $|lk(D)|=\Phi (K_2)$. 
\label{prop-lk-Phi}
\end{proposition}
\medskip 

\begin{proof}
Let $D=K_1 \cup K_2$ be an oriented diagram such that $K_1$ has no self-crossings. 
If $K_2$ has either no non-self over-crossings or non-self under-crossings, then $D$ represents a split link by Proposition~\ref{prop-sp} and we obtain $lk(D)=0=\Phi (K_2)$. 

Assume that $K_2$ has both non-self over- and under-crossings. 
Take an over-crossing $p$ of $K_2$ and a base point $b$ on $K_2$ just before $p$. 
By traversing $K_2$ from the base point $b$, we obtain a non-cyclic sequence $S$ starting with $O$. 
Note that $S$ corresponds to $\hat{f}(K_2)$ by regarding $S$ as a cyclic sequence. 

Let $R$ be the region bounded by the simple closed curve $K_1$ that does not contain the base point $b$. 
When we traverse $K_2$ from $b$, we alternately enter and exit the region $R$ at the non-self crossings. 
Hence, each pair of non-self crossings whose corresponding letters in $S$ are in the same parity position have the same orientation of $K_2$ relative to $K_1$. 
For example, see Figure \ref{fig-lk-P} for a diagram with $S=OUUUOU$, or $S=O_1 U_2 U_3 U_4 O_5 U_6$ with labels for convenience. 
The orientation of $K_2$ at a non-self crossing that corresponds to a letter in the odd (resp. even) position in $S$ has upward (resp. downward) orientation. 
\begin{figure}[ht]
\centering
\includegraphics[width=10cm]{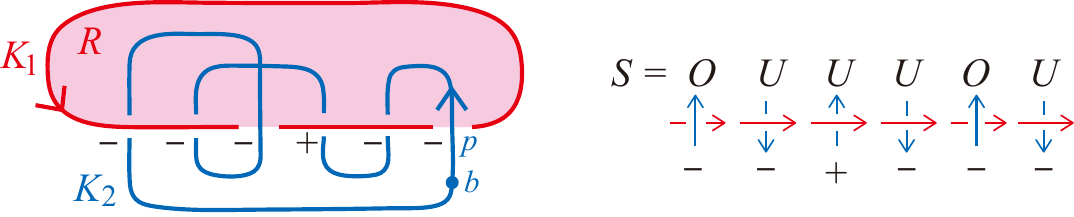}
\caption{A diagram $D=K_1 \cup K_2$ with $\hat{f}(D)=(w_1, w_2)=(OOUOOU, OUUUOU)$. 
With the indicated base point $b$, we have the non-cyclic sequence $S=OUUUOU$. 
We observe that $\Phi (S)= lk (D)= -2$. }
\label{fig-lk-P}
\end{figure}

Suppose that $p$ is a positive (resp.\ negative) crossing. 
Then each crossing corresponding to ``$O$'' in an odd position or ``$U$'' in an even position in $S$ is positive (resp.\ negative), while each crossing corresponding to ``$O$'' in an even position or ``$U$'' in an odd position is negative (resp.\ positive). 
Therefore, $lk(D)=- \Phi (S)$ (resp.\ $lk(D)= \Phi (S)$) by the definition of $\Phi (S)$ (see Figure \ref{fig-lk-P}).
Hence, we have $|lk(D)|= |\Phi (S) | = \Phi (K_2)$. 
\end{proof}
\medskip 

\noindent The following corollary follows immediately from the proof of Proposition \ref{prop-lk-Phi}. 

\medskip 
\begin{corollary}
Let $D=K_1 \cup K_2$ be a link diagram such that all the non-self crossings are on the single edge of the diagram $K_1$. 
Then $|lk(D)| = \Phi (K_2)$. 
\label{cor-k1-k2}
\end{corollary}
\medskip 

\noindent For diagrams that have no self-crossings of both $K_1$ and $K_2$, we have the following corollary. 

\medskip 
\begin{corollary}
If a diagram $D=K_1 \cup K_2$ has no self-crossings, then $\Phi (K_1)= \Phi (K_2)= | lk(D) |$. 
\label{cor-no-self-crossing}
\end{corollary}
\medskip 

\begin{example}
See the diagrams $D_1$, $D_2$, $D_3$ of the link $6^2_1$ in Figure \ref{fig-T26}. 
The diagram $D_1$ has no self-crossings and $\Phi (K^1_1)= \Phi (K^1_2)=3= | lk(6^2_1)|$. 
For $D_2$, we can see that the component $K^2_2$ has a self-crossing from the OU number $\Phi (K^2_1)=1 \neq |lk (6^2_1)|$ by Corollary \ref{cor-k1-k2}. 
For $D_3$, we can see that both components have a self-crossing from $\Phi (K^3_1)= \Phi (K^3_2)=1 \neq |lk(6^2_1)|$. 
\begin{figure}[ht]
\centering
\includegraphics[width=9cm]{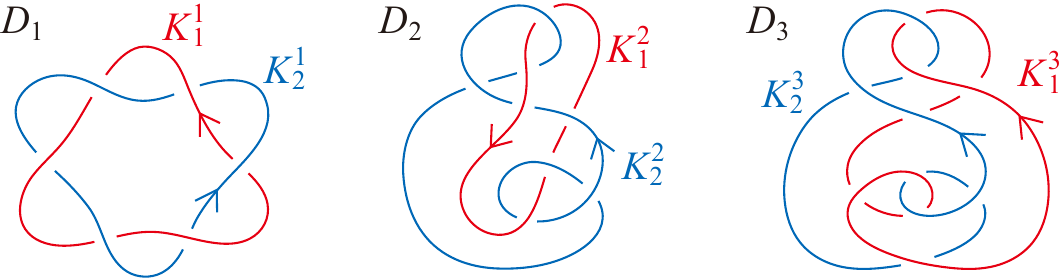}
\caption{Oriented diagrams $D_1$, $D_2$, $D_3$ of the link $6^2_1$ with $\hat{f}(D_1)=(OUOUOU, OUOUOU)$, $\hat{f}(D_2)=(OOOUUU, OUOUOU)$, $\hat{f}(D_3)=(OOOUUU, OOOUUU)$. }
\label{fig-T26}
\end{figure}
\end{example}
\medskip

\noindent Now we prove Proposition \ref{prop-self-crossing}. \\

\noindent {\it Proof of Proposition \ref{prop-self-crossing}.} \ 
It follows from the contrapositive of Corollary \ref{cor-no-self-crossing}. 
\hfill$\square$  \\

\noindent For a given well-balanced pair $(w_1, w_2)$, we have the following corollary from Lemma \ref{lem-1} and Proposition \ref{prop-lk-Phi}. 

\medskip 
\begin{corollary}
For a well-balanced pair $(w_1, w_2)$ of cyclic sequences of $O$ and $U$, there are link diagrams $D=K_1 \cup K_2$ such that $\hat{f}(D)=(w_1, w_2)$ and $lk(D)=\Phi (w_1)$ or $\Phi (w_2)$. 
\end{corollary}
\medskip 

\begin{proof}
Let $D_1$ (resp. $D_2$) be a link diagram obtained by Construction 1 (resp. 2) on $(w_1, w_2)$. 
By Proposition \ref{prop-lk-Phi}, we have $|lk(D_1)|= \Phi (K_2)= \Phi (w_2)$, $|lk(D_2)|= \Phi (K_1)= \Phi (w_1)$. 
By reversing the orientation of one component if necessary, we obtain a diagram satisfying the condition. 
\end{proof}

\subsection{Link diagrams with alternating sequences}

In this subsection, we study link diagrams and their linking numbers with alternating non-self OU sequences. 
We have the following lemma. 

\medskip 
\begin{lemma}
Let $w$ be an alternating cyclic sequence of $O$ and $U$ with even length $2m$. 
For any number $n$ with $-m \leq n \leq m$ and $n \equiv m \pmod{2}$, there exists a diagram $D=K_1 \cup K_2$ with $\hat{f}(D)=(w,w)$ and $lk(D)=n$. 
\label{lem-mod2}
\end{lemma}
\medskip 

\begin{proof}
Let $D= K_1 \cup K_2$ be the minimal diagram of the $(2,2m)$-torus link with the orientation indicated in Figure \ref{fig-m2}. 
Then $D$ satisfies $\hat{f}(D)=(w, w)$ and $lk(D)= \Phi (w) = -m$. 
\begin{figure}[ht]
\centering
\includegraphics[width=9cm]{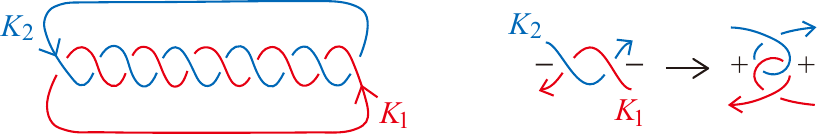}
\caption{A diagram with linking number $-m=-5$ on the left-hand side and a transformation on the right-hand side.}
\label{fig-m2}
\end{figure}
Apply the transformation depicted on the right-hand side of Figure \ref{fig-m2} for $\frac{m+n}{2}$ times at some bigons where $K_1$ is upper. 
Note that $\frac{m+n}{2}$ is an integer since $n \equiv m \pmod{2}$ and the value of $\frac{m+n}{2}$ does not exceed the number of the bigons where $K_1$ is upper since $n \leq m$. 
By this transformation, the non-self OU sequence is preserved and the linking number increases by 2. 
Thus, we obtain a diagram $D'$ such that $\hat{f}(D')=(w,w)$ and $lk(D')=n$. 
\end{proof}
\medskip 

\noindent We obtain the following corollary, regarding a core part of a diagram in the proofs of Lemmas \ref{lem-trivial} and \ref{lem-hopf}. 

\medskip 
\begin{corollary}
Let $w$ be an alternating sequence of $O$ and $U$ of even length. 
\begin{itemize}
\item[(1)] If $\# O(w) \equiv 0 \pmod{2}$, then there exists a diagram $D$ of a trivial link with $\hat{f}(D)=(w,w)$. 
\item[(2)]  If $\# O(w) \equiv 1 \pmod{2}$, then there exists a diagram $D$ of a Hopf link with $\hat{f}(D)=(w,w)$. 
\end{itemize}
\label{cor-diagrams}
\end{corollary}

\begin{proof}
Diagrams are obtained in the same way as the proof of Lemma \ref{lem-mod2} by applying the transformation shown in Figure \ref{fig-m2} to every other bigons where $K_1$ is upper. 
Such diagrams, shown in Figure \ref{fig-T}, are denoted by $A(N)$, $B(N)$, where $N= \lfloor \frac{\# O(w)}{2} \rfloor$. 
Note that the tangle $T$ depicted on the left-hand side of Figure \ref{fig-T} represents a trivial tangle and therefore $A(N)$ (resp. $B(N)$) represents a trivial link (resp. a Hopf link). 
\begin{figure}[ht]
\centering
\includegraphics[width=12cm]{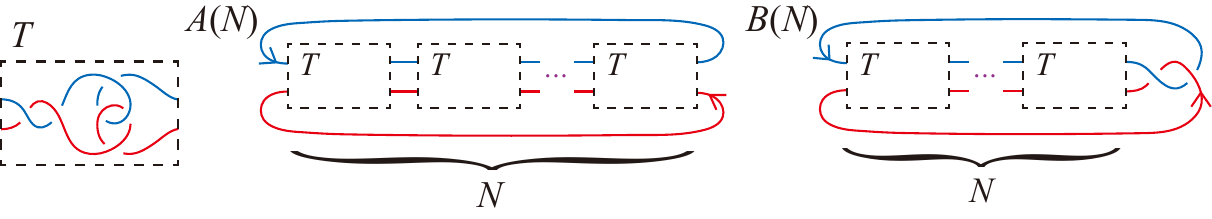}
\caption{Diagrams $A(N)$ representing a trivial two-component link and $B(N)$ representing a Hopf link. }
\label{fig-T}
\end{figure}
\end{proof}
\medskip 

\noindent We remark that $A(N)$ (resp. $B(N)$) has OU number $2N$ (resp. $2N+1$).

\section{Proof of Theorem \ref{thm-ch2}}
\label{section-proof-ch2}

In this section, we prove Theorem \ref{thm-ch2}.

\begin{lemma}
If $(w_1, w_2)$ is well-balanced and $\# O(w_1) \equiv 0 \pmod{2}$, then there exists a diagram $D=K_1 \cup K_2$ of the trivial link with $\hat{f}(D)=(w_1, w_2)$. 
\label{lem-trivial}
\end{lemma}
\medskip 

\noindent Before the proof, let us see an example of constructing a diagram $D$ of a trivial link such that $\hat{f}(D)=(w_1, w_2)$ for a given well-balanced pair $(w_1, w_2)$ with $\# O(w_1) \equiv 0 \pmod{2}$, for the case of the OU number $\Phi (w_1)=0$. 

\medskip 
\begin{example}
Let $(w_1, w_2)=(OUUO, OUOU)$, which has OU numbers $\Phi (w_1)=0$ and $\Phi (w_2) =2$. 
We can deform $w_1$ into the empty sequence $\emptyset$ by a sequence, say $r$, of two reductions, as $O\boldsymbol{UU}O \to \boldsymbol{OO} \to \emptyset$. 
Take a non-cyclic sequences $S_1$ and $S_2$ from $w_1$ and $w_2$ by choosing any initial letter, and assign the letters subscripts representing its position in the sequence. 
For example, let $S_1=O_1 U_2 U_3 O_4$, $S_2 = O_1 U_2 O_3 U_4$. 

For $S_1$, assign the letters superscripts representing the order of reduction in the sequence $r$, namely, assign as $S_1 = O^2_1 U^1_2 U^1_3 O^2_4$. 
For $S_2$, then, assign the letters superscripts in the following manner. 
If $S_1$ has $O^i_a$ and $O^i_b$ (resp. $U^j_c$ and $U^j_d$) for some $a, b$ (resp. $c, d$), assign the superscript $i$ (resp. $j$) to any two $U$s  (resp. $O$s) in $S_2$. 
Namely, let $S_2=O^1_1 U^2_2 O^1_3 U^2_4$. 
The superscripts in $S_1$ and $S_2$ indicate the pairing of non-self crossings of a diagram we construct. 

Draw two oriented simple closed curves $K_1$, $K_2$ on $S^2$ as shown in the upper-left of Figure~\ref{fig-Phi-zero}. 
Take four points on $K_1$, $K_2$ and assign them the letters according to $S_1$, $S_2$, respectively. 
We transform the diagram $K_1 \cup K_2$ in the following procedure (a)--(d). 
\begin{figure}[ht]
\centering
\includegraphics[width=12cm]{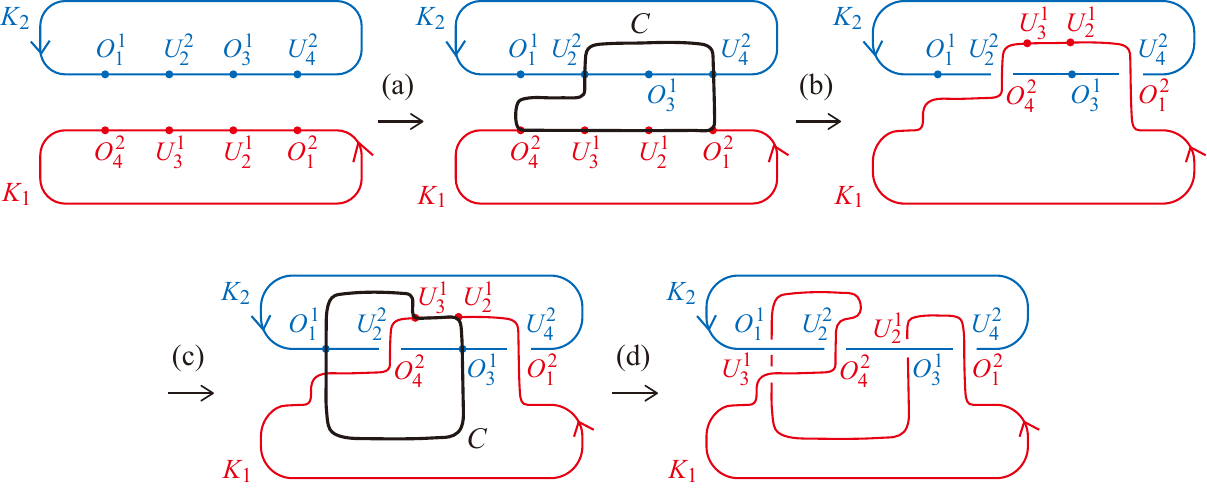}
\caption{Procedure (a)--(d) of constructing a diagram $D$ of a trivial link with $\hat{f}(D)=(OUUO, OUOU)$}
\label{fig-Phi-zero}
\end{figure}
\begin{enumerate}[label=(\alph*)]
\item Take a portion $I$ of $K_1$ that has endpoints $O^2_1$ and $O^2_4$. 
Draw a simple closed curve $C$ on $S^2$ so that $C$ overlaps only $I$ and crosses transversely only the points $U^2_2$, $U^2_4$ on $K_2$. 
\item On $K_1$, replace $I$ with the closure $\overline{C \setminus I}$. 
For the new crossings on $C \setminus I$, give crossing information so that $C \setminus I$ is over. 
Assign the letters $O^2_1$ and $O^2_4$ to the new non-self crossings and retake the points $U^1_2$, $U^1_3$ on $K_1$ between $O^2_1$ and $O^2_4$ according to $S_1$. 
\item Take the portion $I$ of $K_1$ that has no non-self crossings and has endpoints $U^1_2$, $U^1_3$. 
Draw a simple closed curve $C$ so that $C$ overlaps only $I$ and crosses only the points $O^1_1$, $O^1_3$ transversely on $K_2$. 
The curve may crosses $K_1$ if necessary. 
\item On $K_1$, replace $I$ with $\overline{C \setminus I}$. 
For the new crossings on $C \setminus I$, assign crossing information so that $C \setminus I$ is under. 
Assign the letters $U^1_2$ and $U^1_3$ to the new non-self crossings according to $S_1$. 
\end{enumerate}
We observe that the obtained diagram $D$ represents a trivial link since the above procedure does not change the link type of the diagram. 
Thus, we obtain a diagram $D$ of a trivial link with $\hat{f}(D)=(w_1, w_2)$. 
\label{ex-Phi-zero}
\end{example}
\medskip 

\noindent Example~\ref{ex-Phi-zero} demonstrates a case of $\Phi (w_1)=0$. 
For the case that the OU number is non-zero, we cannot obtain an empty sequence $\emptyset$ by reductions. 
Instead, we obtain a non-empty alternating sequence. 
Utilizing the core diagram $A(N)$ given in Corollary~\ref{cor-diagrams}(1), then, we prove Lemma~\ref{lem-trivial}. 

\medskip 
\begin{proof}[Proof of Lemma \ref{lem-trivial}]
Let $(w_1, w_2)$ be a pair of OU sequences that is well-balanced and satisfies $\# O(w_1) \equiv 0 \pmod{2}$. 
Without loss of generality, we assume that $\Phi (w_1) \leq \Phi (w_2)$; we may change the order of $w_1, w_2$ and $K_1, K_2$ if $\Phi (w_1) > \Phi (w_2)$. 
Note that $\Phi (w_1) \equiv \Phi (w_2) \equiv \# O(w_1) \equiv 0 \pmod{2}$ by Proposition~\ref{prop-Phi-O}. 
Suppose that $w_1$ can be transformed into an alternating sequence by $k$ reductions, namely, $k= \frac{1}{2} (|w_1|-2 \Phi (w_1))$. 

We obtain a diagram $D=K_1 \cup K_2$ of the trivial link with $\hat{f}(D)=(w_1, w_2)$ in the following 5 steps (see Example \ref{ex-ex}). 

\vspace{1ex}
\noindent\textbf{Step 1.}
For $w_1$, fix $\Phi (w_1)$ pairs of ``$OU$'' that remain in the alternating sequence $w'_1$ after applying a sequence, say $r$, of $k$ reductions. 
This is feasible by Lemma~\ref{lem-seq-1}(2). 
For $w_2$, fix any $\Phi (w_1)$ pairs of ``$OU$''. 
This is feasible by $\Phi (w_1) \leq \Phi (w_2)$ and Lemma \ref{lem-seq-1} (1). 

\vspace{1ex}
\noindent\textbf{Step 2.}
Take non-cyclic sequences $S_1$ and $S_2$ from $w_1$ and $w_2$ so that a fixed ``$OU$'' pair is at the head if it exists. 
For $S_1$ and $S_2$, assign a subscript to each letter representing its position. 

Give unfixed letters in $S_1$ superscripts representing the order of reductions in the sequence $r$. 
Then give unfixed letters in $S_2$ superscripts in the following manner. 
If there is a pair $O^i_a$ and $O^i_b$ in $S_1$ for some $a, b$, then give the superscript $i$ to two $U$s in $S_2$. 
If there is a pair $U^j_c$ and $U^j_d$ in $S_1$ for some $c, d$, then give the superscript $j$ to two $O$s in $S_2$. 
This assignment is feasible because $(w_1, w_2)$ is well-balanced. 

\vspace{1ex}
\noindent\textbf{Step 3.}
Take non-cyclic alternating sequences $S'_1$ and $S'_2$ from $S_1$ and $S_2$ by deleting the unfixed letters. 
Note that $\Phi (S'_i)=\Phi (S_i)= \Phi (w_i) \equiv 0 \pmod{2}$ for $i=1,2$. 
Let $A ( \frac{\Phi (w_1)}{2} )$ be the oriented diagram of a trivial link that is illustrated in Figure \ref{fig-T} with $N=\frac{\Phi (w_1)}{2}$. 
Assign each non-self crossing of $A ( \frac{\Phi (w_1)}{2} )$ a pair of letters $O$ and $U$ with subscripts according to $S'_1$ and $S'_2$. 

\vspace{1ex}
\noindent\textbf{Step 4.}
Take points on arcs of the diagram $A ( \frac{\Phi (w_1)}{2} )$ and assign them the rest letters $O$ and $U$ with subscripts and superscripts according to $S_1$ and $S_2$. 

\vspace{1ex}
\noindent\textbf{Step 5.}
Apply the following procedure for $l=k, k-1, \dots , 2, 1$:

\noindent
(a) \textit{Construction}:
Take a portion $I$ of $K_1$ between $O^l_a$ and $O^l_b$ or $U^l_c$ and $U^l_d$ which does not contain non-self crossings. 
It is possible to take such $I$ because there are no fixed ``$OU$'' pairs between $O^l_a$ and $O^l_b$ or $U^l_c$ and $U^l_d$ in $S_1$. 
Draw a simple closed curve $C$ on $S^2$ that overlaps only $I$, crosses transversely only the two points of $U^l_\alpha$ and $U^l_\beta$ or $O^l_\gamma$ and $O^l_\delta$ on $K_2$. 
The curve $C$ may cross $K_1$ transversely if necessary. 
It is possible to draw such $C$ since $K_2$ can be regarded as a circle by regarding the hooks derived from $A ( \frac{\Phi (w_1)}{2} )$ that tangle with $K_1$ very small. 
Note that $U^l_\alpha$ and $U^l_\beta$ or $O^l_\gamma$ and $O^l_\delta$ are not in the hooks. 

\noindent
(b) \textit{Assignment}:
On $K_1$, replace the portion $I$ with $\overline{C \setminus I}$. 
For the new crossings on $C \setminus I$, give the crossing information so that $C \setminus I$ is over (resp. under) if $S_1$ has $O^l_a$ (resp. $U^l_c$). 
Assign the letters $O^l_a$ and $O^l_b$ or $U^l_c$ and $U^l_d$ to the new non-self crossings, and retake points and assign the letters that were on $I$ to $C \setminus I$ according to $S_1$. 

\vspace{2ex}
Thus, we obtain a diagram $D=K_1 \cup K_2$ with $\hat{f}(D)=(w_1 , w_2)$. 
Moreover, $D$ represents the trivial link since each procedure of Step 5 does not change the link type of the diagram. 
Indeed, we can reduce the crossings of $D$ that are produced in the procedures by Reidemeister moves for $l=1, 2, \dots , k$, and we obtain the diagram $A ( \frac{\Phi (w_1)}{2} )$ again. 
\end{proof}
\medskip

\begin{example}
Let $w_1=OOUOUOUU$, $w_2= OUOUOUOU$. 
Then $(w_1, w_2)$ is well-balanced, $\# O(w_1)=4 \equiv 0 \pmod{2}$, $\Phi (w_1)=2$, $\Phi (w_2)=4$ and $k=2$. 
\begin{itemize}
\item[Step 1.] Fix two ``$OU$'' pairs in $w_1$ as follows: 
$w_1=OOU(OU)(OU)U$, where we consider a sequence $r$ of reductions on $w_1$ as $\boldsymbol{OO}UOUOUU \to \boldsymbol{U}OUOU \boldsymbol{U} \to OUOU$. 
Then fix two ``$OU$'' pairs in $w_2$ as follows: $w_2=(OU)OU(OU)OU$. 
\item[Step 2.] Take $S_1$, $S_2$, and assign subscripts and superscripts to the letters in $S_1$, $S_2$ as follows: 
$S_1=(O_1 U_2)(O_3 U_4) U^2_5 O^1_6 O^1_7 U^2_8$, $S_2= (O_1 U_2) O^2_3 U^1_4 (O_5 U_6) O^2_7 U^1_8$. 
\item[Step 3.] Then we have $S'_1=O_1 U_2 O_3 U_4$, $S'_2 = O_1 U_2 O_5 U_6$ and we use the diagram $A(1)$. 
\begin{figure}[ht]
\centering
\includegraphics[width=8.5cm]{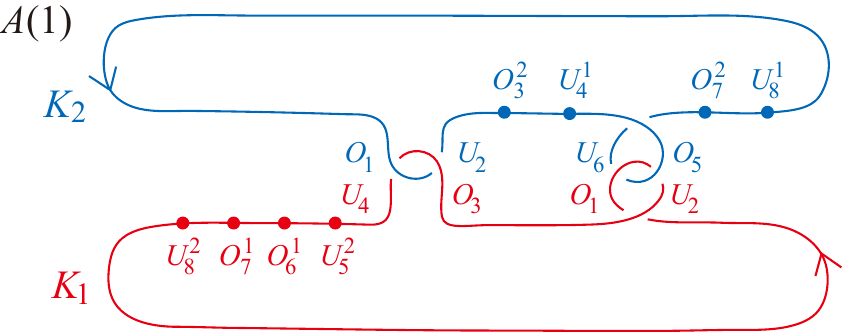}
\caption{The diagram $A(1)$ equipped with points with the letters in $S_1$, $S_2$.}
\label{fig-ex-3}
\end{figure}
\item[Step 4.] Assign the crossing information to $A(1)$ as shown in Figure \ref{fig-ex-3}. 
\item[Step 5.] Perform the transformation of Step 5 for $l=2$ and then $l=1$ so that the points $U^2_5$, $U^2_8$ on $K_1$ and $O^2_3$, $O^2_7$ on $K_2$ match, and then $O^1_6$, $O^1_7$ on $K_1$ and $U^1_4$, $U^1_8$ on $K_2$, as shown in Figure \ref{fig-ex-5}. 
\begin{figure}[ht]
\centering
\includegraphics[width=12cm]{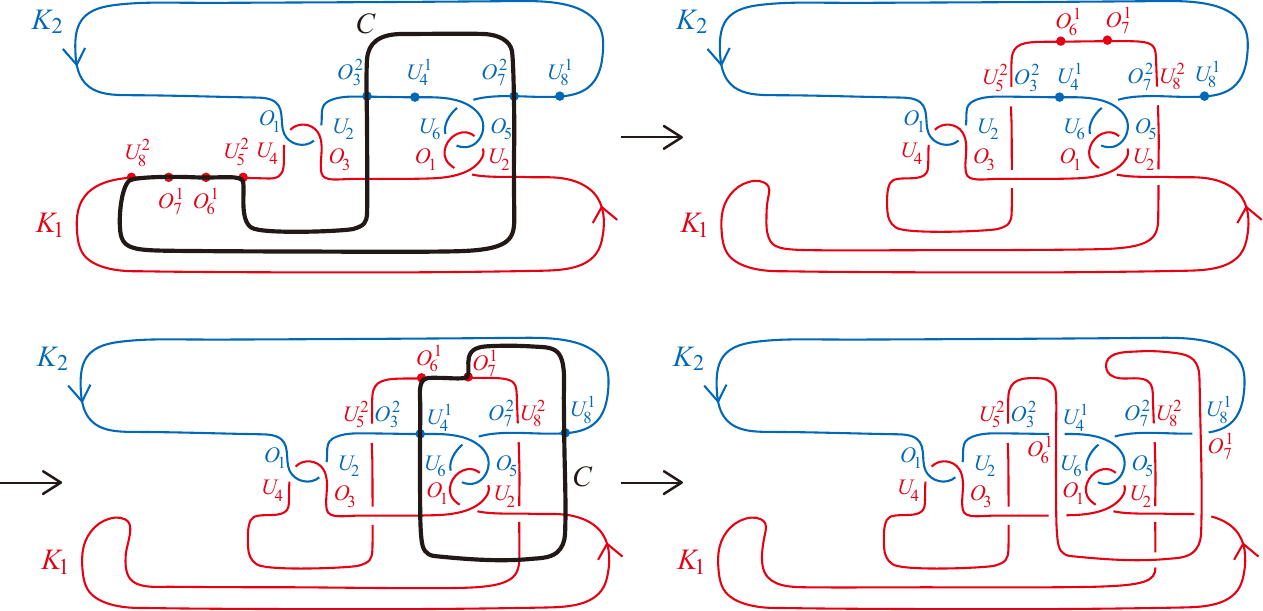}
\caption{Step 5 for $l=2, 1$. }
\label{fig-ex-5}
\end{figure}
\end{itemize}
Thus, we obtain a diagram $D$ of the trivial link with $\hat{f}(D)=(w_1, w_2)$. 
We can see that $D$ is deformed into the diagram $A(1)$ by Reidemeister moves by looking at the diagrams in Figure \ref{fig-ex-5} in opposite order. 
\label{ex-ex}
\end{example}
\medskip 

\noindent We prove Theorem \ref{thm-ch2}. 

\medskip 
\begin{proof}[Proof of Theorem \ref{thm-ch2}]
It follows from Theorem \ref{thm-ch} (II) and Lemma \ref{lem-trivial}.
\end{proof}
\medskip

\section{Proof of Theorem \ref{thm-ch3}}
\label{section-proof-ch3}

In this section, we prove Theorem~\ref{thm-ch3}. 
Similarly to Lemma~\ref{lem-seq-1}, we prove the following lemma, which will be used in Lemmas~\ref{lem-T24} and~\ref{lem-C}, regarding the ``$OU$'' or ``$UO$'' pairs that survive reductions.

\medskip 
\begin{lemma}
Let $w$ be a cyclic sequence of $O$ and $U$ of even length with $\# O(w) \geq 2$, $\# U(w) \geq 2$. 
\begin{itemize}
\item[(1)] In $w$, there are disjoint pairs of ``$OU$ and $UO$'' or ``$OU$ and $OU$''. 
\item[(2)] Suppose that $\Phi (w)=0$. 
There is a sequence of reductions to transform $w$ into a sequence $w'$ of length 4 such that $w$ has disjoint pairs ``$OU$'' and ``$UO$'' that remain in $w'$. 
\end{itemize}
\label{lem-seq-2}
\end{lemma}

\begin{proof}
Let $w$ be a cyclic sequence of $O$ and $U$ of even length with $\# O(w) \geq 2$, $\# U(w) \geq 2$. 

\vspace{1ex}
\noindent
(1)
By the condition, $w$ has at least one ``$OU$'' pair. 
If the next letter is ``$O$'' (resp. ``$U$''), then $w$ has an ``$OU$'' pair (resp. ``$UO$'' pair) after the ``$OU$'' pair. 

\vspace{1ex}
\noindent
(2)
Suppose that $\Phi (w)=0$. 
Let $m$ be the length of $w$. 
Note that $m$ is even and $m \geq 4$. 
We prove the claim (2) by an induction on $m$. 
When $m=4$, this holds since $w=OOUU=OUUO$. 

Assume that the claim holds when $m=l$. 
Suppose that $m=l+2$. 
Then $w$ has a subsequence $v$ of two or more consecutive $O$s or $U$s. 
Let $w^*_1$ (resp. $w^*_2$) be the sequence obtained from $w$ by a reduction at the first and second (resp. second and third if exist) letters in $v$. 

By assumption, both $w^*_1$ and $w^*_2$ have disjoint ``$OU$'' and ``$UO$'' pairs that survive a sequence of reductions to a sequence of length 4. 
The original sequence $w$ has the same ``$OU$'' and ``$UO$'' pairs as at least one of the sequences $w^*_1$ and $w^*_2$. 
\end{proof}

\begin{example}
Let $w= U_1 O_2 O_3 O_4 O_5 U_6$. 
By a reduction in the consecutive $O$s, we obtain $w^*_1 = U_1 O_4 O_5 U_6$ and $w^*_2 = U_1 O_2 O_5 U_6$. 
In this case, $w$ has the same ``$OU$'' and ``$UO$'' pairs to $w^*_2$. 
\end{example}
\medskip 

\noindent For the Hopf link $2^2_1$, we show the following lemma. 

\medskip 
\begin{lemma}
If $(w_1, w_2)$ is well-balanced and $\# O(w_1) \equiv 1 \pmod{2}$, then there exists a diagram $D$ of a Hopf link with $\hat{f}(D) =(w_1, w_2)$. 
\label{lem-hopf}
\end{lemma}
\medskip 

\begin{proof}
It can be proved in the same way as the proof of Lemma \ref{lem-trivial}, where we use the core diagram $B ( \frac{\Phi (w_1)-1}{2} )$ shown in Figure \ref{fig-T} in Step~3 instead of $A ( \frac{\Phi (w_1)}{2} )$. 
\end{proof}
\medskip 

\noindent For the Solomon link $4^2_1$, we show the following lemma. 

\medskip 
\begin{lemma}
If $(w_1, w_2)$ is well-balanced and $\# O(w_1) \equiv 0 \pmod{2}$, $\# O(w_1) \geq 2$, and $\# U(w_1) \geq 2$, then there exists a diagram $D$ of a Solomon link with $\hat{f}(D)=(w_1, w_2)$. 
\label{lem-T24}
\end{lemma}
\medskip 

\begin{proof}
We assume that $\Phi (w_1) \leq \Phi (w_2)$ without loss of generality. 
\begin{itemize}
\item[$\bullet$] If $\Phi (w_1) \geq 2$, it can be proved in the same way as the proof of Lemma \ref{lem-trivial}, where we use the core diagram $C( \frac{\Phi (w_1)-2}{2})$ shown in Figure \ref{fig-C} in Step 3 instead of $A( \frac{\Phi (w_1)}{2})$. 
\begin{figure}[ht]
\centering
\includegraphics[width=5.5cm]{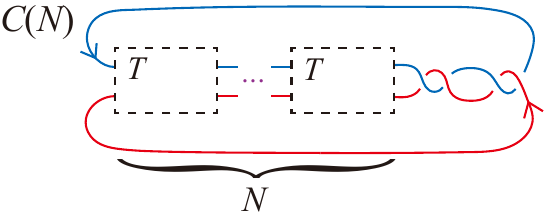}
\caption{A diagram $C(N)$ representing a Solomon link. The tangle $T$ is illustrated in Figure \ref{fig-T}. }
\label{fig-C}
\end{figure}
\item[$\bullet$] If $\Phi (w_1)=0$, it can be proved in a similar way to the proof of Lemma \ref{lem-trivial}, where we execute the following Steps $2'$ and $3'$ instead of Steps 2 and 3. 
\begin{itemize}
\item[Step $2'$.] By Lemma \ref{lem-seq-2} (2), $w_1$ has disjoint pairs ``$OU$'' and ``$UO$'' that survive a sequence of reductions $r$ to length four. 
Fix such ``$OU$'' and ``$UO$'' in $w_1$. 
Fix ``$OU$ and $UO$'' or ``$OU$ and $OU$'' in $w_2$ (Lemma \ref{lem-seq-2} (1)). 
\item[Step $3'$.] Take a diagram $D^a$ (resp. $D^b$) of a Solomon link in Figure \ref{fig-ab-24} if $S'_2=(OU)(UO)$ (resp. $S'_2=(OU)(OU)$). 
\begin{figure}[ht]
\centering
\includegraphics[width=7cm]{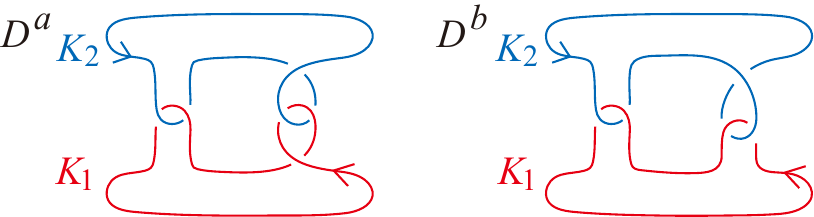}
\caption{Diagrams $D^a$, $D^b$ representing a Solomon link with a pair of non-self OU sequences $(UOOU, OUUO)$, $(UOOU, OUOU)$, respectively.}
\label{fig-ab-24}
\end{figure}
\end{itemize}
\end{itemize}
\end{proof}

\noindent For the Whitehead link $5^2_1$, we show the following lemma. 

\medskip 
\begin{lemma}
If $(w_1, w_2)$ is well-balanced and $\#O(w_1) \equiv 0 \pmod{2}$, $\# O(w_1) \geq 2$ and $\# U(w_1) \geq 2$, then there exists a diagram $D= K_1 \cup K_2$ of a Whitehead link with $\hat{f}(D)=(w_1, w_2)$. 
\label{lem-C}
\end{lemma}
\medskip 

\begin{proof}
\begin{figure}[ht]
\centering
\includegraphics[width=5.5cm]{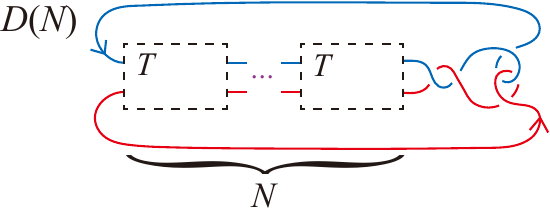}
\caption{A diagram $D(N)$ representing a Whitehead link. The tangle $T$ is illustrated in Figure \ref{fig-T}.}
\label{fig-D}
\end{figure}
It can be proved in the same way as the proof of Lemma \ref{lem-T24}, where we use the diagram $D( \frac{\Phi (w_1)-2}{2})$ in Figure \ref{fig-D} in Step 3 instead of $C( \frac{\Phi (w_1) -2}{2})$, and use the diagrams $D^{a'}$, $D^{b'}$ in Figure \ref{fig-ab-C} in Step $3'$ instead of $D^a , D^b$. 
\begin{figure}[ht]
\centering
\includegraphics[width=7.5cm]{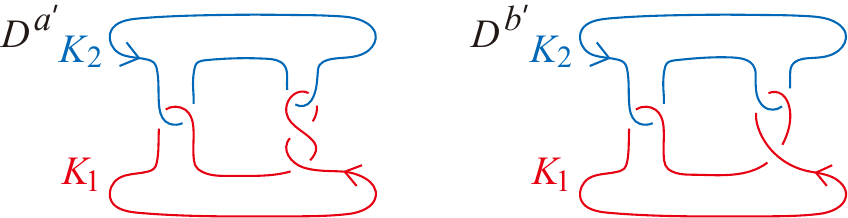}
\caption{Diagrams $D^{a'}$, $D^{b'}$ representing a Whitehead link with a pair of non-self OU sequences $(UOOU, OUUO)$, $(UOOU, OUOU)$, respectively.}
\label{fig-ab-C}
\end{figure}
\end{proof}
\medskip

\noindent We prove Theorem \ref{thm-ch3}. 

\medskip 
\begin{proof}[Proof of Theorem \ref{thm-ch3}]
Statement (A) follows from Theorem \ref{thm-ch} (III) and Lemma \ref{lem-hopf}. 
Statement (B) follows from Theorem \ref{thm-ch} (II), Proposition \ref{prop-sp} and Lemmas \ref{lem-T24}, \ref{lem-C}. 
\end{proof}
\medskip

\noindent From Theorems \ref{thm-ch}, \ref{thm-ch2}, \ref{thm-ch3} and Proposition \ref{prop-sp}, we obtain the following corollary. 

\medskip 
\begin{corollary}
Let $\mathcal{L}^S$ (resp. $\mathcal{L}^N$) be the set of all oriented diagrams of two-component split (resp. non-split) links. 
The following holds. 
\begin{itemize}
\item[$\bullet$] $\hat{f}(\mathcal{L}^S)= \hat{f}(\mathcal{L}^{\text{even}})$. 
\item[$\bullet$] $\hat{f}(\mathcal{L}^N)= \{  (w_1, w_2) \ | \ (w_1, w_2) \text{: well-balanced, } 
\#O(w_1), \  \#U(w_1) \geq 1 \}$.
\end{itemize}
\label{cor-non-split}
\end{corollary}

\section*{Acknowledgment}
This work was partially supported by the JSPS KAKENHI Grant Number JP21K03263 and JST CREST, Japan, Grant Number JPMJCR25Q3.

\end{document}